\pdfoutput=1
\documentclass[oneside,a4paper,11pt]{article}
\usepackage{mathtools}
\usepackage{amsmath}
\allowdisplaybreaks[4]
\usepackage{csquotes}
\usepackage{amsthm}
\usepackage{amssymb}
\usepackage{mathrsfs}
\usepackage{color}
\usepackage{bm}
\usepackage{comment}
\usepackage{fancyhdr}
\usepackage{geometry}
\usepackage{hyperref}
\usepackage{enumerate}
\usepackage{graphicx}
\usepackage{subfig}
\usepackage{float}
\usepackage{array}

\theoremstyle{definition}
\newtheorem{thm}{Theorem}[section]
\newtheorem{cor}[thm]{Corollary}
\newtheorem{lem}[thm]{Lemma}
\newtheorem{prop}[thm]{Proposition}
\newtheorem{defn}[thm]{Definition}

\numberwithin{equation}{section}

\title{Asymptotic behavior of discrete Schr\"{o}dinger equation on the hexagonal triangulation}
\author{Huabin Ge, Bobo Hua, Longsong Jia, Puchun Zhou
}

\date{}
\usepackage{fancyhdr}
\begin{document}
	\maketitle
	\begin{abstract}
	 In this article, we prove the decay estimate for the discrete Schr\"odinger equation (DS) on the hexagonal triangulation. The $l^1\rightarrow l^\infty$ dispersive decay rate is $\left\langle t\right\rangle^{-\frac{3}{4}}$, which is faster than the decay rate of  DS on the 2-dimensional lattice $\mathbb{Z}^2$, which is $\left\langle t\right\rangle^{-\frac{2}{3}}$, see \cite{stefanov2005asymptotic}. The proof relies on the detailed analysis of singularities of the corresponding phase function and the theory of uniform estimates on oscillatory integrals developed by Karpushkin \cite{kar_paro}. Moreover, we prove the Strichartz estimate and give an application to the discrete nonlinear Schr\"odinger equation (DNLS) on the hexagonal triangulation.
   \\
		
	\end{abstract}
	
\section{Introduction} \label{sec:1}
\subsection{discrete Schr\"odinger equation on the hexagonal triangulation}
Discrete Schr\"odinger equations are fundamental discrete dynamical models, which has many applications in physics, see e.g. \cite{morandotti1999dynamics,onorato2001freak,kopidakis2001targeted,burger2002quasi,christodoulides2003discretizing,livi2006self}.
Let $G=(V,E)$ be an undirect graph with weights $\{w_{e}\}_{e\in E}$ on edges. We write $u\sim v$ if $u$ and $v$ are connected by an edge in $E$.
  A Schr\"odinger equation on $G$ describes the continuous-time random walk (CTQW) on $G$, see e.g. \cite{salimi2008study}, written as
   \begin {eqnarray}
   \left\{
   \begin{aligned}
   \label{LS}&i\partial_tu(v,t)+\Delta_Gu(v,t)=0,~~\forall (v,t)\in V\times[0,\infty),\\
   &u(v,0)=\psi(v),~~\forall v\in V.
   \end{aligned}
   \right.
   \end{eqnarray}
   where the discrete Laplacian $\Delta_G$ is given by
   \begin{align}
   \Delta_Gf_u=\sum_{v:v\sim w}\omega_{uv}(f_v-f_w).
   \end{align}

The Strichartz estimate is a powerful tool for analyzing the asymptotic behavior of dispersive equations, which was first introduced in \cite{strichartz1977restrictions}. 
To establish the Strichartz estimate, one usually need an $l^1\rightarrow l^\infty$ estimate (or $L^1\rightarrow L^{\infty}$ estimate in the continuous case) with the form
\begin{align}
\|u(v,t)\|_{l^\infty}\le C\left\langle t\right\rangle^{-\sigma}\|\psi\|_{l^1},\label{form}
\end{align}
where $\left\langle t\right\rangle$ is the Japanese bracket given by $\left\langle t\right\rangle:=1+|t|, $ and $C$ is a constant independent of $t.$ In the continuous case, it is well-known that the $L^1\rightarrow L^\infty$ estimate for linear Schr\"odinger equations in $\mathbb{R}^d$ holds with $\sigma=\frac{d}{2}.$ As for the discrete Schr\"odinger equations on $\mathbb{Z}^d$, the $l^1\rightarrow l^{\infty}$ estimate holds for $\sigma=\frac{d}{3}$, see \cite{stefanov2005asymptotic}. Application of those dispersive estimates can be found in \cite{MR2468536,MR2578796,MR2877358,MR3117519}, where the spectral problem and stability of breathers are discussed. Dispersive estimates for the discrete wave equation (DW) and the discrete Klein-Gorden equation (DKG) on $\mathbb{Z}^d$ are also studied in \cite{stefanov2005asymptotic,MR2529950,MR3433284,MR4331254,bi2023wave}. 

 Dispersive estimates on metric graphs has also been studied in the past 20 years.\ In  \cite{adami_starshape,banica2011dispersion,banica2014dispersion,ignat2010strichartz,mehmeti2015dispersive,mehmeti2017dispersive} the dispersive estimates for star-shaped networks and tadpole graph are discussed. Recently, dispersive estimates on regular trees and Cartesian product of the integer lattice and a finite graph were studied by Ammari and Sabri \cite{MR4645063,MR4651402}. After that finite metric graphs with infinite ends are also considered, see \cite{mehmeti2023dispersive}. We should also remark that the dispersive estimate of quantum walk in discrete time on $1-$dimensional lattice is established in \cite{maeda2022dispersive}.
 
  Those works inspired us to explore dispersive estimates on more general graphs, especially the triangulations of surfaces. In this article, we mainly focus on the discrete Schr\"odinger equation (DS) on the $1$-skeleton of the hexagonal triangulation $\mathcal{T}_H$, which is a regular triangulation of $\mathbb{R}^2$ with degree $6$, as shown in Figure \ref{hexagon}. The weights we focus on here are the simplest weights with $\omega\equiv 1.$ We denote by $\Delta_{G_H}$ the discrete Laplacian on $\mathcal{T}_H$ . 
  \begin{figure}[htbp]
  	\centering
  	\includegraphics[scale=0.3]{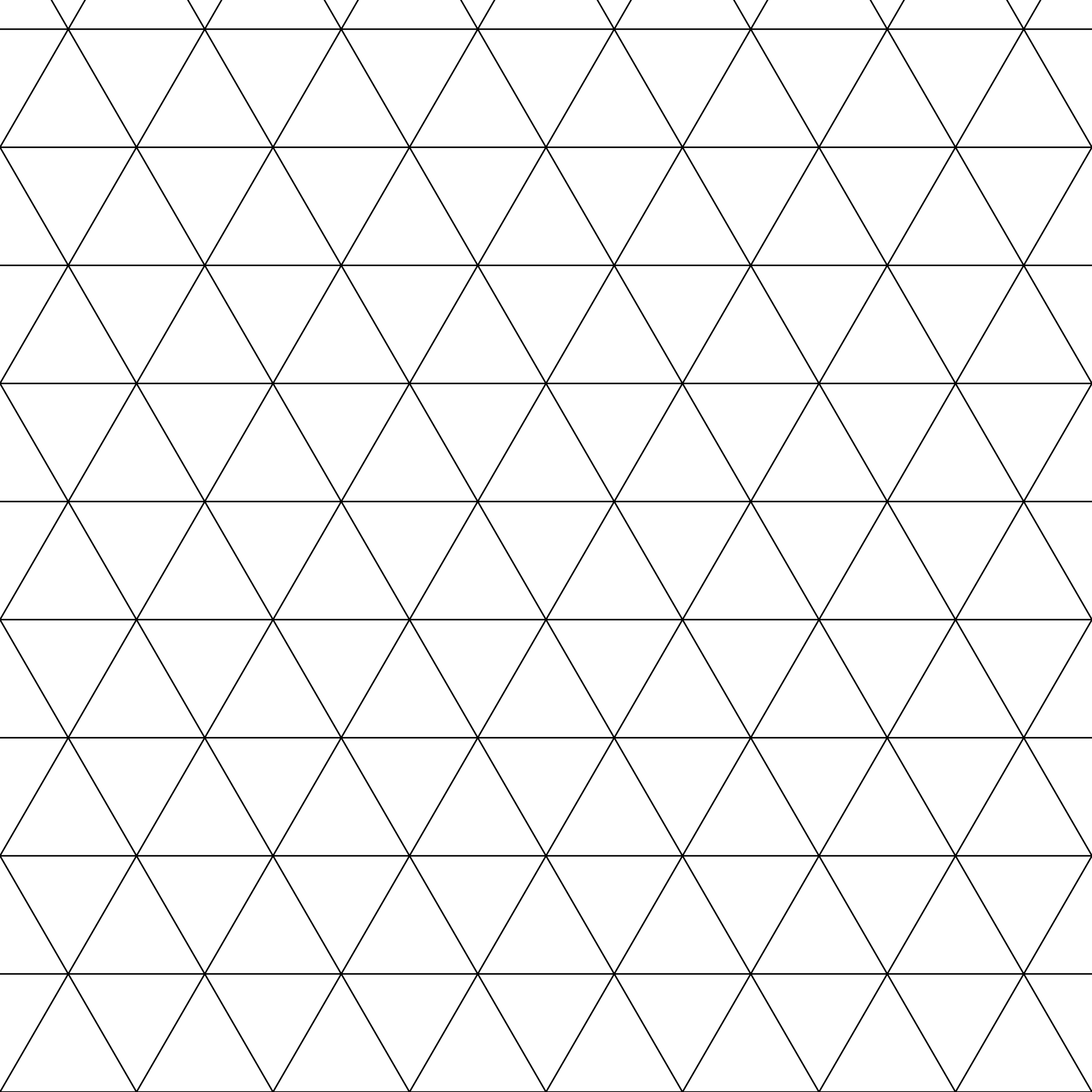}
  	\captionof{figure}{\small Hexagonal triangulation $\mathcal{T}_H$ of the plane.}
  	\label{hexagon}
  \end{figure} 
  
Since the hexagonal triangulation can be viewed as a Cayley graph of the two-dimensional lattice $\mathbb{Z}^2$, 
with the help of discrete Fourier analysis on the lattice, we can obtain the explicit formula of the solution to \eqref{LS} in $C^1([0,+\infty);l^2(\mathbb{Z}^2))$.
\begin{align}\label{formula}
    u(n,t)=2\sum_{k\in \mathbb{Z}^2}u(k,0)\int_{[0,2\pi]^2}e^{-itg(x)+i\left\langle(k-n),x\right\rangle}\mathrm{d}x,~~\forall n\in \mathbb{Z}^2,
\end{align}
where $x=(x_1,x_2)^t$ and 
\begin{align}\label{phasefunction}
    g(x)=6-2\cos x_1-2\cos x_2-2\cos(x_1+x_2).
\end{align}

To obtain the $l^1\rightarrow l^{\infty}$ estimate of the solution to \eqref{LS}, it is sufficient to obtain a uniform estimate of the oscillatory integral 
$G(l,t)=\int_{[0,2\pi]^2}e^{-itg(x)+i\left\langle l,x\right\rangle}\mathrm{d}x$, i.e.
\begin{align}\label{osc_est}
G(l,t)\le C\left\langle t\right\rangle^{-\sigma},~~\forall~l\in \mathbb{Z}^2,
\end{align}
where $C$ is a positive constant independent of the parameter $l$.

The estimate of oscillatory integrals required for the Schr\"odinger equation on the hexagonal triangulation differs from the case on $\mathbb Z^2$ in \cite{stefanov2005asymptotic}, as the inability to apply the method of separation of variables poses significant difficulties for the estimate.
In this article, we will apply the method in \cite{MR1611132}, which established the dispersive estimate of the discrete wave equation on $\mathbb
{Z}^2$, and reduce the estimate \eqref{osc_est} to a uniform estimate of a  oscillatory integral with parameters.

Let $v=(v_1,v_2)=l/t$, then we only need to estimate the oscillatory integral
\[
I(t,v)=G(tv,t)=\int_{[0,2\pi]^2}e^{-it\phi(v,x)}\mathrm{d}x,~~(t,v)\in \mathbb{R}\times \mathbb{R}^2,
\]
where the phase function $\phi(v,x)$ is given by 
\begin{align}\label{phase_core}
\phi(v,x)=g(x)-\left\langle v,x\right\rangle.
\end{align}
The problem then is reduced to find a uniform estimate of $I(t,v).$ First, we transform $I(x,t)$ to an oscillatory integral on $\mathbb{R}^2.$ Let $\mathbf{a}(x)$ be a function in $C_c^{\infty}(\mathbb{R}^2)$, whose support is contained in $[-4\pi,4\pi]^2.$ Moreover, $\mathbf{a}(x)$ is chosen to have the following properties.
\begin{enumerate}
    \item There exists a neighborhood $W$ of $[0,2\pi]^2$ such that \[\mathbf{a}(x)\neq 0,~~\forall x\in W.\]
    \item $\sum_{l\in \mathbb{Z}^2}\mathbf{a}(x+2\pi l)=1,~~\forall x\in \mathbb{R}^2.$
\end{enumerate}
Then we have
\begin{align}
I(t,v)&=\sum_{l\in \mathbb{Z}^2}\int_{[0,2\pi]^2}e^{-it\phi(x)}\mathbf{a}(x+2\pi l)\mathrm{d}x\nonumber\\&=\int_{\mathbb{R}^2}e^{-it\phi(x)}\mathbf{a}(x)\mathrm{d}x\label{final_form}
\end{align}

With the help of uniform estimates of oscillatory integrals in Section \ref{sec:2} and the analysis of singularities of the phase function $\phi(x)$ in Section \ref{sec:4}, we get the following lemma.
\begin{lem}\label{mainlem}
    Let $u$ be the solution to \eqref{LS}, then we have the following $l^1\rightarrow l^{\infty}$ decay estimate.
    \begin{align}
        \|u(\cdot,t)\|_{l^\infty}\le C\left\langle t\right\rangle^{-\frac{3}{4}}\|\psi\|_{l^1}
    \end{align}
\end{lem}

Then by the theory of Keel and Tao \cite{keel1998endpoint}, we can finally obtain the desired Strichartz estimate.

Let us first introduce the notion of admissible pairs, see \cite[Chapter 2]{tao2006nonlinear}.
\begin{defn}
    Fix $\sigma>0$, we call a pair  $(q,r)$ of exponents $\sigma$-admissible if $2\le q,r \le \infty$, $\frac{1}{q}+\frac{\sigma}{r}=\frac{\sigma}{2}$, and $(q,r,\sigma)\neq(2,\infty,1).$
\end{defn}

With the notion of the admissible pair, we have the following theorem.
\begin{thm}\label{mainthm}
Let $u(v,t)\in C^1([0,+\infty);l^2(\mathbb{Z}^2))$ be a solution to $\eqref{LS}$ on the hexagonal triangulation $\mathcal{T}_H$. Then one has the energy identity
\begin{align}\label{energy}
    \|u(\cdot,t)\|_{l^2}=\|\psi\|_{l^2},
\end{align}
and the sharp decay estimate 
\begin{align}\label{decay}
\|u(\cdot,t)\|_{l^{\infty}}\le C\left\langle t\right\rangle^{-\frac{3}{4}}\|\psi\|_{l^1}.
\end{align}

And for the inhomogeneous discrete Schr\"odinger equation
\[
i\partial_tu(v,t)+\Delta_{G_H}u(v,t)+F(v,t)=0,~~\forall (v,t)\in V\times[0,\infty),
\]
there is the Strichartz estimates with $\sigma$-admissible pairs $(q,r)$ and $(\tilde{q},\tilde{r})$, where $\sigma=\frac{3}{4}$.
\begin{align}\label{strichartz}   \|u\|_{L^ql^r}\le C\left(\|\psi\|_{l^2}+\|F\|_{L^{\tilde{q}'}(0,T)l^{\tilde{r}'}}\right),
\end{align}
where $T\in(0,\infty]$ and $C$ is a constant independent of $T$.
\end{thm}

    From Theorem \ref{decay} we see that the decay rate of the discrete Schr\"odinger equation on the hexagonal triangulation is faster than that on $\mathbb{Z}^2$. This phenomenon could be attributed to two possible reasons. Firstly, on the hexagonal triangulation, in addition to edges in the horizontal and vertical directions, there are edges in an additional direction, resulting in a faster decay rate of the solution. Secondly, compared to $\mathbb{Z}^2$, geodesic balls on the hexagonal triangulation are closer to those in Euclidean space. Therefore, the decay rate is more similar to that of the Schr\"odinger equation on a plane.

    It is interesting to consider the estimate of Schr\"odinger equation on other cellular decompositions of $\mathbb{R}^2$, such as the hexagonal tiling.

\section{Uniform estimates of oscillatory integrals}\label{sec:2}
In this section, we introduce some results of uniform estimates of oscillatory integrals, which are used to prove Lemma \ref{mainlem}. The results we use come from the pioneer work of uniform estimates of phase functions with two variables introduced by Karpushkin \cite{MR0778884}.
\begin{defn}
    An oscillatory integral with phase $\phi$ and amplitude $\mathbf{a}$ is an integral given by
    \begin{align}\label{general_integral}
            J(t,\phi,\mathbf{a})=\int_{\mathbb{R}^n}\mathbf{a}(x)e^{it\phi(x)}\mathrm{d}x,
    \end{align}
where $\mathbf{a}\in C_0^{\infty}(\mathbb{R}^n)$ and $t\in\mathbb{R}.$ If the support of $\mathbf{a}$ is sufficiently small which is contained in a small neighborhood of $x=0$, and $\phi$ is an analytic function at the origin, then as $t\rightarrow\infty$, one can obtain the following expansion, see \cite{MR0778884}.
\begin{align}
J(t,\phi,\mathbf{a})\approx\sum_s\sum_{k=0}^{n-1}b_{s,k}(\mathbf{a})t^{s}(\ln t)^k,\label{expansion}
\end{align}
where $s$ is a negative rational number independent of $\mathbf{a}.$
\end{defn}
\begin{defn}
    Let $(\beta(\phi),p(\phi))$ be the maximum over all pairs $(s,k)$ in \eqref{expansion} under the  lexicographic ordering with the following property.
    
    \textbf{For any neighborhood $W$ of $0$, there exists a function $\mathbf{a}(x)\in C_c^{\infty}(W)$, such that $b_{\beta(\phi),p(\phi)}(\mathbf{a})\neq 0.$}
    
     We denote by $O(\phi)=(\beta(\phi),p(\phi))$  the maximal exponential pair.\ $\beta(\phi)$ and $p(\phi)$ are called the oscillatory index of $\phi$ and the multiplicity of the oscillation of $\phi$
at $0$, see \cite{newtonestimate,kar_paro1}.
\end{defn}
Now we assume that $f:\mathbb{R}^n\rightarrow\mathbb{R}$ is analytic at the point $x_0.$ For a positive number $r>0$, we denote by $H_{r,x_0}$ the set of real analytic functions on $B(x_0,r)$ which has a unique holomorphic extension on the ball $B_{\mathbb{C}^n}(x_0,r).$ 

By $H_{r,x_0}(\epsilon)$, we denote the set given by
\[
H_{r,x_0}(\epsilon)=\{f\in H_{r,x_0}:|f|\le \epsilon.\}
\]
\begin{defn}
    An oscillatory integral with phase $\phi$ is said to have a uniform estimate at the point $x_0$ with exponent $(\beta,p)$
     if for any sufficiently small number $r>0$, there exists $\epsilon=\epsilon(r)>0$, $C=C(r)>0$ and a neighborhood $W\subset B(x_0,r)$ such that 
     \[
     J(t,\phi+P,\mathbf{a})\le C\left\langle t\right\rangle^{\beta}\ln^p(|t|+2)\|\mathbf{a}\|_{C^N(W)},~~\forall t\in \mathbb{R}.
     \]
     given that $\mathbf{a}\in C_c^{\infty}(W) $ and $P\in H_{r,x_0}(\epsilon).$
     And the constant $N$ only depends on the dimension $n.$
     For simplicity, we denote by  
     \[
     M(\phi,x_0)\curlyeqprec(\beta,p),
     \]
      that $\phi$ has a uniform estimate at the point $x_0$ with exponent $(\beta,p).$
\end{defn}
In \cite{MR0778884}, it was proved that for the function $\phi:\mathbb{R}^n\rightarrow\mathbb{R}$ is analytic at $x_0\in\mathbb{R}^n$ with $|\nabla\phi(x_0)|=0$, and $\mathrm{Hess}\phi(x_0)$ has corank 2, then one has $M(\phi,x_0)\curlyeqprec O(\phi).$ In particular, when $n=2$, the following uniform estimates hold.  
 \[
M(\phi,x_0)\curlyeqprec O(\phi_{x_0}),
\]
where $\phi_{x_0}=\phi(x+x_0).$
\begin{thm}[\cite{MR0778884}]
	Let $f:\mathbb{R}^2\rightarrow\mathbb{R}$ be a function that analytic at $0,$ with $\nabla f(0)=0$ and $\mathrm{det}(\mathrm{Hess}f(0))=0.$ Then 
	\[
	M(f,0)\curlyeqprec O(\phi).
	\]
\end{thm}
And by the method of stationary phase, one has the following lemma.
\begin{lem}\label{dimensional_reduction}
	Let $m,n\ge1$, and 
	\[
	h(\xi,y)=h_0(\xi)+P(y),~\forall (\xi,y)\in \mathbb{R}^n\times\mathbb{R}^m,
	\]
	where $P(y)=\sum_{j=1}^mc_jy_j^2$ with $c_j\neq 0.$ If $M(h_0,0)\curlyeqprec (\beta,p)$, then we have
	\[
	M(h)\curlyeqprec (\beta-\frac{m}{2},p).
	\]
\end{lem}
\section{Newton Polyhedra}\label{sec:3}
In this section, we introduce some basics of Newton polyhedra, see \cite{MR1484770}. Let $\phi$ be a function on $\mathbb{R}^n$ which is real analytic at $0$. Assume that
\begin{align}
	\phi(0)=|\nabla\phi(0)|=0, \label{assumption}
\end{align}
and the Taylor series of $\phi$ at $0$ is
\begin{align*}
	\phi(x)=\sum_{\alpha\in \mathbb{Z}^n}a_{\alpha}x^{\alpha},
\end{align*}
set $\mathcal{T}(\phi):=\{\alpha\in\mathbb{Z}^n:a_{\alpha}\neq 0\}$, which is called the Taylor support of $\phi$ at $0$. The \textbf{Newton polyhedron} $N(\phi)$ is the convex hull of the set
\begin{align*}
	\cup_{\alpha\in \mathcal{T}(\phi)}(\alpha+(\mathbb{R}_+)^n).
\end{align*}
\begin{defn}\label{non_degenerate}
Let $\mathcal{P}$ be a face of $\mathcal{N}(\phi).$ We denote by $\phi_{\mathcal{P}}(x)$  the function defined as
\begin{align*}
\phi_{\mathcal{P}}(x)=\sum_{\alpha\in \mathcal{P}}a_\alpha x^\alpha.
\end{align*}
We call $\phi$ an $\mathbb{R}$-nondegenerate function if for any compact face $\mathcal{P}$ of $\mathcal{N}(\phi)$, the following relationship holds.
\begin{align}\label{Rdegenerate}
    \cap_{i=1}^n\{x:\partial_i \phi_{\mathcal{P}}(x)=0\}\subset\cup_{i=1}^n\{x:x_i=0\}.
\end{align}
Otherwise, we call $\phi$ is $\mathbb{R}$-degenerate.
\end{defn}
The \textbf{Newton distance} $d_{\phi}$ is given by
\begin{align*}
	d_{\phi}=\inf\{\rho>0:(\rho,...,\rho)\in \mathcal{N}(\phi)\}.
\end{align*}
Set $\textbf{d}_\phi=(d_\phi,...,d_\phi)$. Let $F_\phi$ be the principal face of $\mathcal{N}_{\phi}$, which is the face of minimal dimension containing $\textbf{d}_\phi.$ Set $k_\phi=n-\mathrm{dim}(F_\phi).$ 
Then Varchenko proved the following result in \cite{newtonestimate}.
\begin{thm}\label{newton_lemma}
	Let $\phi$ be $\mathbb{R}-nondegenerate$ and $\phi(0)=|\nabla\phi(0)|=0.$ Assume that $(\beta(\phi),p(\phi))$ is the maximal exponential pair. Then one has
	\begin{align*}
		\beta(\phi)\le -\frac{1}{d_\phi},~p(\phi)\le k_\phi-1.
	\end{align*}
\end{thm}

\section{Singularities of the phase function and the  $l^1\rightarrow l^\infty$ estimate}\label{sec:4}
Now let us analyze the oscillatory integral $I(t,v)$ defined in \eqref{final_form}. Since the phase function $\phi(v,x)=g(x)-\left\langle v,x\right\rangle$, and 
\begin{align*}
\nabla g(x)&=(2\sin(x_1)+2\sin(x_1+x_2),2\sin(x_2)+2\sin(x_1+x_2))^t,
\end{align*}
it is easy to see that $\|\nabla g(x)\|\le 4\sqrt{2}.$
Then when $\|v\|=\|\frac{l}{t}\|\ge 4\sqrt{2}+1$, we have $\|\nabla\phi(v,x)\|\ge 1$. Therefore, with the help of the method of stationary phase, see \cite{stein1993harmonic}, we can easily obtain that the oscillatory integral $I(v,t)$ decays faster than any algebraic function of $t$ uniformly. So it is sufficient to consider the decay estimate when the parameter $v$ is contained inside the ball $B(0,4\sqrt{2}+1).$

Let us first consider the Hessian of the phase function $\mathrm{Hess}_x\phi(v,x)$. It is simply to verify that
\begin{align}\label{hess}
\mathrm{Hess}_x\phi(v,x)=2\begin{pmatrix}
    \cos(x_1+x_2)+\cos x_1&\cos(x_1+x_2)\\
    \cos(x_1+x_2)&\cos(x_1+x_2)+\cos{x_2}
    \end{pmatrix}.
\end{align}
We observe that the corank of the Hessian of $\phi(v,x)$ is at most $1$. Otherwise, we have $\cos(x_1+x_2),\cos{x_1}$ and $\cos{x_2}$ are all equal to $0$. Therefore, $x_i=k_i\pi+\frac{\pi}{2}$, for some integers $k_i,~i=1,2$. Then $x_1+x_2=(k_1+k_2+1)\pi$, which leads to a contradiction.

Since the Hessian of $\phi(v,x)$ is independent of the parameter $v$, we can define a set $\Gamma$ consisting of points $x\in\mathrm{spt}(\mathbf{a})$ where $\mathrm{Hess}_x\phi(v,x)$ degenerate, where $\mathrm{spt}(\mathbf{a})$ is the support of function $a$ in the definition \eqref{final_form}.
\begin{align*}
    \Gamma=\{x\in\mathrm{spt}(\mathbf{a}):\mathrm{det}(\mathrm{Hess}_x)=0\}.
\end{align*}
For simplicity, from now on we denote by $\Omega$ the support of the function $\mathbf{a}$. 
For each $v\in \nabla g(\Gamma),$ there exists a set $$\Gamma(v)=\{x\in\Omega:\nabla_x \phi(v,x)=0,\mathrm{det}(\mathrm{Hess}_x\phi(v,x))=0.\}$$
It is easy to see that for each fixed $v\in\nabla g(\Gamma)$, $\Gamma(v)$ is a nonempty compact set. Then by the argument in Section \ref{sec:2}, for each $\xi\in \Gamma(v)$, there exists an exponential pair $(\beta(\phi(v,\cdot),\xi),p(\phi(v,\cdot),\xi))$ such that
\[
M(\phi(v,\cdot),\xi)\curlyeqprec(\beta(\phi(v,\cdot),\xi),p(\phi(v,\cdot),\xi)).
\]
Then by the definition of uniform estimates for oscillatory integrals, for each $v\in \nabla g(\Gamma)$, there exists an exponential pair $(\beta(v),p(v))$ and positive constants $\epsilon(v)$ and $C(v)$ such that 
\[
I(v',t)\le C(v)\left\langle t\right\rangle^{\beta(v)}\log|t|^{p(v)},~~\forall v'\in B(v,\epsilon(v)).
\]
Since $\nabla g(\Gamma)$ is a compact subset of $\mathbb{R}^2,$ and $\cup_vB(v,\epsilon(v))\supset\nabla g(\Gamma)$, there exists a finite collection $\{v_i\}_{i=1}^N\subset \nabla g(\Gamma)$ such that  
\[
\cup_{i=1}^NB(v_i,\epsilon(v))\supset\nabla g(\Gamma).
\]
Therefore one can get a uniform estimate of $I(v,t)$ when $v\in \cup_{i=1}^NB(v_i,\epsilon(v))$. For rest the parameters $v\in B(0,4\sqrt{2}+1)\backslash\cup_{i=1}^NB(v_i,\epsilon(v))$, it is easy to find that $\mathrm{det}(\mathrm{Hess}_x)\phi(v,x)\ge C_0$ for some positive constant $C_0,$ therefore by the stationary phase method, one can obtain that 
\[
I(v,t)\le C_1\left\langle t\right\rangle^{-1},~~\forall v\in B(0,4\sqrt{2}+1)\backslash\cup_{i=1}^NB(v_i,\epsilon(v)),
\]
where $C_1$ is a positive constant independent of $v.$

Before analyzing singularities of the phase function $\phi(v,x)$, let us first introduce a useful lemma. We call $\alpha\in\mathbb{R}^n$ a \textbf{weight} if $\alpha$ is a vector that consists of positive elements. By $k^\alpha x$ we denote $(k^{\alpha_1}x_1,...,k^{\alpha_n}x_n)$ for every $k>0$ and $x\in \mathbb{R}^n.$
\begin{defn}\label{homogeneous}
For a polynomial $h:\mathbb{R}^n\rightarrow\mathbb{R}$, we call $h$ \textbf{quasi-homogenous} of degree $\rho$ with respect to the weight $\alpha$, if and only if $h(k^{\alpha} x)=k^\rho h(x).$ Let $\mathcal{E}_{\alpha,n}$ be the sets of all quasi-homogenous polynomials of degree $1$ with respect to $\alpha$. And by $\mathcal{H}_{\alpha,n}$ we denote the set of functions which is real-analytic at the origin, whose Taylor series consists of monomials that are quasi-homogeneous of degree greater than $1$ with respect to $\alpha.$ 
\end{defn}
Given a positive weight $\alpha$ and a function $h$ that is real-analytic at the origin, then the following lemma, which was first proved in \cite{kar_paro}, is crucial for our analysis.
\begin{lem} \label{principle}
    Let $h:\mathbb{R}^n\rightarrow\mathbb{R}$ be a function which is real-analytic at the origin. Assuming that there exists a polynomial $h_0(x)\in \mathcal{E}_{\alpha,n}$ and an analytic function $h_1(x)\in \mathcal{H}_{\alpha,n}$ such that
    \[
    h(x)=h_0(x)+h_1(x),
    \]
    then $M(h_0,0)
    \curlyeqprec(\beta,p)$ implies $M(h,0)\curlyeqprec(\beta,p).$ 
\end{lem}

With the help of Lemma \ref{principle}, we can reduce the problem of estimating the decay rate of some oscillatory with some polynomial phases.
In particular, we have the following Lemma for estimating the oscillatory integral $I(t,v).$
\begin{lem}\label{lemest}
    Let $\phi(v,x)$ be the phase function \eqref{phasefunction}. 
    Set $v\in\nabla g(\Gamma)$ and $x\in \Gamma(v).$ Let $\phi_{x_0}(v,x)=\phi(v,x+x_0).$ 
    Then there exists an invertible linear transform $u=T(x)$ on $\mathbb{R}^n$, the Taylor expansion of the phase function $\phi(v,u)$ with respect to $u$ under the transformation $T$ has one of the following forms,
    \begin{align}
        \phi_{x_0}(v,u)&=\phi_{x_0}(v,0)+\alpha_{2,0}u_1^2+\alpha_{1,2}u_1u_2^2+\alpha_{0,4}u_2^4+h_1(u_1,u_2)\label{normal1},\\
        \phi_{x_0}(v,u)&=\phi_{x_0}(v,0)+\beta_{2,0}u_1^2+\beta_{0,3}u_2^3+h_2(u_1,u_2)\label{normal2},
    \end{align}
    where $u=(u_1,u_2)$, $h_1\in \mathcal{H}_{(\frac{1}{2},\frac{1}{4}),2}$, $h_2\in \mathcal{H}_{(\frac{1}{2},\frac{1}{3}),2}$. Moreover, $\alpha_{2,0}u_1^2+\alpha_{1,2}u_1u_2^2+\alpha_{0,4}u_2^4$ is $\mathbb{R}$-nondegenerate in \eqref{normal1}, and    $\beta_{2,0},\beta_{0,3}\neq 0$ in \eqref{normal2}
\end{lem}
By the definition \eqref{homogeneous}, $\alpha_{2,0}u_1^2+\alpha_{1,2}u_1u_2^2+\alpha_{0,4}u_2^4\in\mathcal{E}_{(\frac{1}{2},\frac{1}{4}),2}$ and $\beta_{2,0}u_1^2+\beta_{0,3}u_2^3\in\mathcal{E}_{(\frac{1}{2},\frac{1}{3}),2}$. Therefore, we only need to estimate the oscillatory integrals whose phase functions are $\alpha_{2,0}u_1^2+\alpha_{1,2}u_1u_2^2+\alpha_{0,4}u_2^4$ or $\beta_{2,0}u_1^2+\beta_{0,3}u_2^3$. By Lemma \ref{dimensional_reduction} and \ref{newton_lemma}, the maximal exponential pairs of the $\mathbb{R}$-nondegenerate polynomials are listed as below.
\begin{center}
\begin{tabular}{|c|c|}
\hline
$f$ & $(\beta(f),p(f))$  \\ \hline
$\alpha_{2,0}u_1^2+\alpha_{1,2}u_1u_2^2+\alpha_{0,4}u_2^4$ &  $(-\frac{3}{4},0)$ \\ \hline
$\beta_{2,0}u_1^2+\beta_{0,3}u_2^3$ & $(-\frac{5}{6},0)$ \\ \hline
\end{tabular}
\end{center}

\begin{lem}\label{singular_estimate}
    Let $0$ be an singularity of phase functions $\phi_1$ and $\phi_2$ with the forms \eqref{normal1} and \eqref{normal2} respectively, then the following uniform estimates hold.
    \begin{align}
        M(\phi_1,0)&\curlyeqprec(-\frac{3}{4},0),\label{A_2}\\
        M(\phi_2,0)&\curlyeqprec(-\frac{5}{6},0).\label{D_4}
    \end{align}
\end{lem} 
\begin{proof}[Proof of Lemma \ref{mainlem}]
    This Lemma can be obtained by Lemma \ref{principle}, \ref{lemest} and \ref{singular_estimate} immediately.
\end{proof}

Now we begin to proof Lemma \ref{lemest}. Since we only need to analyze all possible singularities, in later arguments, we always assume $v\in\nabla g(\Gamma)$ and $x\in \Gamma(v).$ 

Since the Hessian matrix $\mathrm{Hess}_x\phi(v,x)$ is degenerate at $x$,  by the previous arguments, its rank is $1.$ Therefore, without loss of generality, we can assume the following equation holds 
\[
\cos(x_1+x_2)=\lambda(\cos(x_1+x_2)+\cos(x_1)),\cos(x_1+x_2)+\cos(x_2)=\lambda\cos(x_1+x_2).
\]
where $\lambda\in \mathbb{R}.$ And a direct calculation shows that the degeneration of the Hessian is equivalent to
\begin{align}
	\mathrm{det}(\mathrm{Hess}_x\phi(v,x))=(\cos x_1+\cos x_2)\cos(x_1+x_2)+\cos x_1\cos x_2=0.\label{eq1}
\end{align}

\begin{proof}[Proof of Lemma \ref{lemest}]
 We divide our argument into four cases.
 \[(\mathrm{A})\cos{x_1}=\cos{x_2}=0,~(\mathrm{B})\lambda=0,~(\mathrm{C}) \lambda> 0~\text{and}~(\mathrm{D})\lambda<0.\] Here in  $(\mathrm{B})\sim(\mathrm{D})$ we assume that $\cos x_1\neq 0.$
\begin{enumerate}[(A)]
	\item When $\cos{x_1}=\cos{x_2}=0,$ then $\sin(x_1+x_2)=0.$ Then the function $\phi_x(v,x')=\phi(v,x+x')$ has the Taylor expansion
	\begin{align*}
		\phi_x(v,x')   &\sim\phi_x(v,0)+\cos(x_1+x_2)(x_1'+x_2')^2-\frac{1}{3}(\sin x_1(x_1')^3+\sin x_2(x_2')^3)\\
		&-\frac{1}{12}\cos(x_1+x_2)(x_1'+x_2')^4+\cdots
	\end{align*}
Set $u=(u_1,u_2)$, where $u_1=x_1'+x_2'$ and $u_2=x_2'.$ Then 
\begin{align*}
	\phi_x(v,u)&\sim \phi_x(v,0)+\cos(x_1+x_2)u_1^2\\&-\frac{1}{3}(\sin x_1 u_1^3-3\sin x_1 u_1^2u_2+3\sin x_1 u_1u_2^2+(\sin x_2-\sin x_1)u_2^3)
	\\
	&-\frac{1}{12}\cos(x_1+x_2)u_1^4+\cdots
\end{align*}
When $\sin x_1\neq\sin x_2$, $0$ is the singularity with the form \eqref{normal1}. When $\sin x_1=\sin x_2$, $0$ is the singularity with the form \eqref{normal2}, where $\alpha_{2,0}=\pm1,~\alpha_{1,2}=\pm1$ and $\alpha_{0,4}=0.$ Then it is a $\mathbb{R}-$nondegenerate polynomial. 
\item Now we consider the second case where $\lambda=0.$ In this case, we have $\cos x_2=\cos(x_1+x_2)=0,$ and
\begin{align*}
	\mathrm{Hess}_x\phi(v,x)=2\begin{pmatrix}
		\cos x_1&0\\
		0&0
	\end{pmatrix}.
\end{align*}
 In this case,
\begin{align*}
    \phi_x(v,x')&\sim\phi_x(v,0)+\cos x_1(x_1')^2-\frac{1}{3}(\sin{x_2}(x_2')^3+\sin{(x_1+x_2)}(x_1'+x_2')^3)
    \\&-\frac{1}{12}\cos x_1(x_1')^4
    +\cdots,
\end{align*}
where $|\cos{x_1}|= 1.$ So when $\sin x_2+\sin(x_1+x_2)\neq 0$, $0$ is the singularity with the form \eqref{normal1}, otherwise $0$ is the singularity with the form \eqref{normal2}.
\item
Now let us come to the third case, where $\lambda> 0$. In this case, we have the following equations.
\begin{align}
    \cos(x_1)&=\frac{1-\lambda}{\lambda}\cos(x_1+x_2),\label{l1}\\
    \cos(x_2)&=(\lambda-1)\cos(x_1+x_2),\label{l2}\\
    \lambda&=-\frac{\cos(x_2)}{\cos(x_1)}\label{l3}.
\end{align}
Similarly for the function $\phi_x(v,x')=\phi(v,x+x')$, one have the following Taylor expansion.
\begin{align*}
&\phi_x(v,x')   \sim \phi_x(v,0)+\cos{(x_1+x_2)}(\frac{1}{\sqrt{\lambda}}x_1'+\sqrt{\lambda} x_2')^2\nonumber \\
   &-\frac{1}{3}\big[\sin(x_1+x_2)(x_1'+x_2')^3+\sin{x_1}(x_1')^3+\sin{x_2}(x_2')^3\big]
\\&-\frac{1}{12}\big[\cos{(x_1+x_2)}(x_1'+x_2')^4+\cos{x_1}(x_1')^4+\cos{x_2}(x_2')^4\big]+\cdots
\end{align*}
By change of variables, set $u=(u_1,u_2)$, where $u_1=\frac{1}{\sqrt{\lambda}}x_1'+\sqrt{\lambda} x_2'$ and $u_2=x_2'$, we obtain
\begin{align}
&\phi_x(v,u)\sim\phi_x(v,0) + \cos{(x_1+x_2)}u_1^2\nonumber
\\&
-\frac{1}{3}\big[\sin{(x_1+x_2)}(\sqrt{\lambda} u_1+(1-\lambda)u_2)^3
+\sin{x_1}(\sqrt{\lambda} u_1-\lambda u_2)^3+(\sin{x_2})u_2^3\big]\nonumber
\\&-\frac{1}{12}\big[
\cos{(x_1+x_2)}(\sqrt{\lambda} u_1+(1-\lambda)u_2)^4+\cos{x_1}(\sqrt{\lambda} u_1-\lambda u_2)^4+(\cos{x_2})u_2^4
\big]\nonumber
\\&
+\cdots\label{expansion1}
\end{align}
We denote by $\alpha_{2,0},\alpha_{1,2},\alpha_{0,3}$ and $\alpha_{0,4}$ the coefficients of monomials $u_1^2,u_1u_2^2,u_2^3$ and $u_2^4$ in \eqref{expansion1} respectively. Then we have
\begin{align*}
    &\alpha_{2,0}=\cos{(x_1+x_2)},
    \\&\alpha_{1,2}=-
    \sin(x_1+x_2)\sqrt{\lambda}(1-\lambda)^2-\sin{x_1}\sqrt{\lambda}\lambda^2
    ,
    \\&\alpha_{0,3}=-\frac{1}{3}\big[
\sin{(x_1+x_2)}(1-\lambda)^3-\sin{x_1}\lambda^3+\sin{x_2}
    \big],
    \\&\alpha_{0,4}=-\frac{1}{12}
    \big[
    \cos{(x_1+x_2)}(1-\lambda)^4+\cos{x_1}\lambda^4+\cos{x_2}
    \big].
\end{align*}
Let $\Sigma_0$ be solutions to the equation \eqref{eq1}. Moreover, we denote by $\Sigma_1$ and $\Sigma_2$ subsets of $\mathbb{R}^2$ where $\alpha_{0,3}=0$ and $\alpha_{2,0}u_1^2+\alpha_{1,2}u_1u_2^2+\alpha_{0,4}u_2^4$ is $\mathbb{R}$-degenerate respectively.
 Then the assertions in Lemma \ref{singular_estimate} in Case (C) hold if $\Sigma_0\cap\Sigma_1\cap\Sigma_2=\emptyset$ for $\lambda>0.$
 Firstly, we consider the set $\Sigma_1$.
Suppose that the coefficient $\alpha_{0,3}=0$. Then we have 
\begin{align}\label{thrid_term}
(1-\lambda)^3\sin{(x_1+x_2)}-\lambda^3\sin{x_1}+\sin{x_2}=0.
\end{align}
By equations \eqref{l3} and \eqref{thrid_term}, we have
\begin{align*}
(\cos{x_1}+\cos{x_2})^3\sin{(x_1+x_2)}+\cos^3 x_2\sin x_1+\cos^3x_1\sin x_2=0,
\end{align*}
which leads to 
\begin{align}\label{cancel_eq1}
    \sin{(x_1+x_2)}\big[(\cos{x_1}+\cos{x_2})^3+\cos^2x_1+\cos^2x_2-\cos{x_1}\cos{x_2}\cos(x_1-x_2)\big]=0.
\end{align}
Therefore, $\Sigma_1=\Sigma_1^1\cup\Sigma_1^2$, where
\[
\Sigma_1^1=\{(x_1,x_2)\in\Gamma(v):x_1+x_2=k\pi,~\text{for some}~k\in \mathbb{Z}\},
\]
and
$\Sigma_1^2$ is the solution to 
\begin{align}\label{eq2}
(\cos{x_1}+\cos{x_2})^3+\cos^2x_1+\cos^2x_2-\cos{x_1}\cos{x_2}\cos(x_1-x_2)=0. 
\end{align}

So if $x\in\Sigma_0\cap \Sigma_1^1$, by equation \eqref{eq1} we have $\cos{x_2}=0,$ which will be attributed to the cases (A) or (B). Therefore, we only need to analyze the set $\Sigma_0\cap\Sigma_1^2\cap\Sigma_2.$

Suppose that the polynomial $\alpha_{2,0}u_1^2+\alpha_{1,2}u_1u_2^2+\alpha_{0,4}u_2^4$ is $\mathbb{R}$-degenerate, then its discriminant is equal to $0$, that is

\begin{align}\label{core_eq1}
   \alpha_{1,2}^2- 4\alpha_{2,0}\alpha_{0,4}=0.
\end{align}

Since
\begin{align*}
    \alpha_{2,0}\alpha_{0,4}=-\frac{\cos{(x_1+x_2)}}{12}
    \big[
    \cos{(x_1+x_2)}(1-\lambda)^4+\cos{x_1}\lambda^4+\cos{x_2}\
    \big],
\end{align*}
by equations \eqref{l1} and \eqref{l2},
\begin{align}\label{LHS}
    &\alpha_{2,0}\alpha_{0,4}=\frac{\cos^2(x_1+x_2)}{4}\lambda(1-\lambda)^2.
\end{align}
Therefore, by \eqref{core_eq1} and \eqref{LHS} we have
\begin{align}\label{core_eq2}
    \frac{\cos^2(x_1+x_2)}{4}\lambda(1-\lambda)^2=\frac{1}{4}
    \big[
    \sin(x_1+x_2)\sqrt{\lambda}(1-\lambda)^2+\sin{x_1}\lambda^{\frac{5}{2}}
    \big]^2.
\end{align}
By equations \eqref{eq1}, \eqref{l3} and \eqref{core_eq2}, we have 
\begin{align}\label{symmetry1}
    \cos^2{x_1}\cos^2x_2=\big[
    \sin{(x_1+x_2)}\frac{(\cos x_1+\cos{x_2})^2}{\cos{x_1}}+\sin{x_1}\frac{\cos^2{x_2}}{\cos{x_1}}\big]^2
    .
\end{align}
Let $x\in \Sigma_0\cap\Sigma_2$, then by \eqref{cancel_eq1} and \eqref{symmetry1}, we have 
\begin{align}\label{cancel_eq2}
    \cos^2{x_1}\cos^2{x_2}(\cos{x_1}+\cos{x_2})^2=(\cos^2x_2\sin{x_1}-\cos^2{x_1}\sin{x_2})^2.
\end{align}
By direct calculation, \eqref{cancel_eq2} is equivalent to the equation
\begin{align*}
	\cos^2{\frac{x_1+x_2}{2}}[\cos^2{x_1}\cos^2{x_2}\cos^2{\frac{x_1-x_2}{2}}-\sin^2{\frac{x_1-x_2}{2}}(\cos^2{\frac{x_1-x_2}{2}}+\sin^2{\frac{x_1+x_2}{2}})^2]=0.
\end{align*}

Since in case (C), $\cos{\frac{x_1+x_2}{2}}\neq 0,$ we denote by by $\Sigma_2'$ the solution to 
\begin{align}\label{cancel_eq3}
\cos^2{x_1}\cos^2{x_2}\cos^2{\frac{x_1-x_2}{2}}-\sin^2{\frac{x_1-x_2}{2}}(\cos^2{\frac{x_1-x_2}{2}}+\sin^2{\frac{x_1+x_2}{2}})^2=0.
\end{align}
Then sets $\Sigma_0,\Sigma_1^2$ and $\Sigma_2'$ are shown in Figure \ref{sigma012}, with curves colored in red, green and black respectively.

\begin{figure}[htbp]
	\centering
	\includegraphics[scale=0.7]{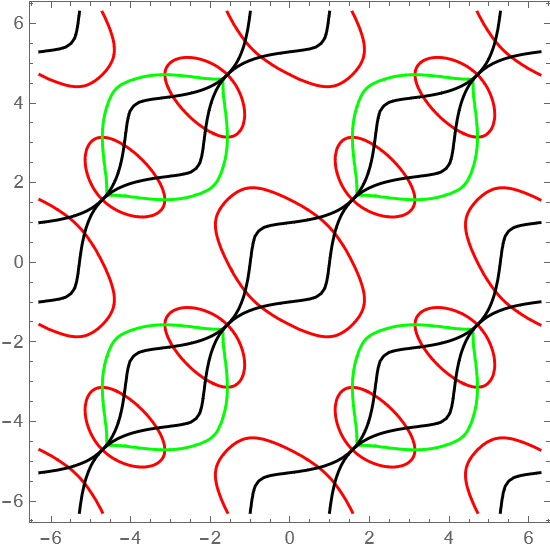}
	\captionof{figure}{\small Curves $\Sigma_0,\Sigma_1^2$ and $\Sigma_2'$.}
	\label{sigma012}
\end{figure} 
The intersection of three curves are empty whenever $\lambda=\frac{\cos{x_2}}{\cos{x_1}}>0$ as shown in Figure \ref{sigma012}. Since the proof consists of lengthy calculations, we will leave it to the Proposition \ref{calculation} in Appendix \ref{app}.
\item 
Finally, we consider the case when $\lambda<0.$ Similarly, let $\phi_x(v,x')=\phi(v,x+x')$, we have
\begin{align*}
&\nonumber \phi_x(v,x')
\sim \phi_x(v,0) -\cos{(x_1+x_2)}(\frac{1}{\sqrt{-\lambda}}x_1'-\sqrt{-\lambda} x_2')^2\nonumber
\\&
-\frac{1}{3}\big[\sin{(x_1+x_2)}(x_1'+x_2')^3+\sin{x_1}(x_1')^3+\sin{x_2}(x_2')^3\big]
\\&
-\frac{1}{12}[\cos(x_1+x_2)(x_1'+x_2')^4+\cos{x_1}(x_1')^4+\cos_{x_2}(x_2')^4
+\cdots\nonumber
\\
&=\phi_x(v,0)-\cos{(x_1+x_2)}u_1^2\\&-\frac{1}{3}\big[\sin{(x_1+x_2)}(\sqrt{-\lambda}u_1+(1-\lambda)u_2)^3\nonumber
+\sin{x_1}(\sqrt{-\lambda}u_1-\lambda u_2)^3+ 
\sin{x_2}u_2^3 \big]
\\&-\frac{1}{12}[\cos(x_1+x_2)(\sqrt{-\lambda}u_1+(1-\lambda)u_2)^4+\cos{x_1}(\sqrt{-\lambda}u_1-\lambda u_2)^4+ 
\cos{x_2}u_2^4]+\cdots~
\end{align*}
where $u_1=\frac{1}{\sqrt{-\lambda}}x_1'-\sqrt{-\lambda} x_2'$ and $u_2=x_2'$. Since all the analyses are similar to the case (C), we omit the detailed proof here.

\end{enumerate}
\end{proof}
\section{Strichartz estimate on hexagonal triangulation and its application}\label{sec:5}
To prove the Strichartz estimate on the hexagonal triangulation, we need the well-known result of Keel and Tao.
\begin{thm}[Keel and Tao \cite{keel1998endpoint}]\label{tao}
Let $H$ be a Hilbert space, $(X,\mathrm{d}x)$ be a measure space, and $U(t):H\to L^2(X)$ be a one-parameter family of mappings, which has the energy estimate
\[
\|U(t)f\|_{L_x^2}\le C\|f\|_H,
\]
and the decay estimate 
\[\|U(t)(U(S))^*g\|_{L^{\infty}(X)}\le C\left\langle t-s\right\rangle^{-\sigma}\|g\|_{L^1(X)},\]
for some $\sigma>0$. Then,
\begin{align*}
    &\|U(t)f\|_{L^qL^r}\le C\|f\|_{L^2},\\
    &\|\int (U(t))^*F(t,\cdot)\mathrm{d}t\|_{L^2}\le C\|F\|_{L_t^{q'}L_x^{r'}},\\
    &\|\int_0^tU(t)(U(s))^*F(s,\cdot)\mathrm{d}s\|_{L^2}\le C\|F\|_{L_t^{\tilde{q}'}L_x^{\tilde{r}'}},
\end{align*}
where exponent pairs $(q,r,\sigma)$ and $(\tilde{q},\tilde{r},\sigma)$ do not equal to $(2,\infty,1)$, and
\[
\frac{1}{q}+\frac{\sigma}{r}\le\frac{\sigma}{2}
.\]
\end{thm}
Now we prove the decay estimate \eqref{decay} of \eqref{LS}. 

\begin{proof}[Proof of Theorem \ref{mainthm}]
    With the help of Theorem \ref{tao} and Lemma \ref{mainthm}, we can prove the Theorem \ref{mainthm} immediately.
\end{proof}

With the help of Theorem \ref{mainthm}, we can establish the global existence of the solution to discrete nonlinear Schr\"odinger equations (DNLS) with a small initial value on the hexagonal triangulation $\mathcal{T}_H.$ 

     \begin{eqnarray}\label{DNLS}
   \left\{\begin{aligned}&i\partial_tu(v,t)+\Delta_{G_H}u(v,t)+|u(v,t)|^{2\sigma}u(v,t)=0,~~\forall (v,t)\in\mathbb{Z}^2\times[0,\infty) \\
    &u(\cdot,0)=\psi\in l^2(\mathbb{Z}^2).
 \end{aligned}
   \right.
   \end{eqnarray}
\begin{thm}
    Set $\sigma\ge \frac{11}{6}$. Then there exists $\epsilon>0$ and a positive constant $C$ such that if $\|\psi\|_{l^2}\le \epsilon$, then \eqref{DNLS} has a global solution $u\in C^1((0,\infty];l^2(\mathbb{Z}^2))$. Moreover,
    \[
    \|\psi\|_{L^ql^r }\le C\epsilon,
    \]
given that
\[
\frac{1}{q}+\frac{3}{4r}\le \frac{3}{8}.
\]
As a corollary, we have 
\[
\lim_{t\rightarrow\infty}\|u(t)\|_{l^r}=0,~~\forall r>2.
\]
\end{thm}
\begin{proof}
    Define a metric space
    \[
    \mathcal{M}=\{u(t):\sup_{q,r}\|u\|_{L^ql^r}<2C\|\psi\|_{l^2},~\text{where}~\frac{1}{q}+\frac{3}{4r}\le\frac{3}{8}\},
    \]
    where $C$ is the constant in Strichartz estimate \ref{mainthm}.
    Let $L$ be an operator such that
    \[
    L(u)=e^{it\Delta_{G_H}}\psi+i\int_0^te^{i(t-s)\Delta_{G_H}}|u|^{2\sigma}u(s)\mathrm{d}s.
    \]
    Let the exponent pair $(\tilde{q},\tilde{r})=(\infty,2),$ by \eqref{strichartz} we have
    \begin{align*}
        &\sup_{(q,r):\frac{1}{q}+\frac{3}{4r}\le\frac{3}{8}}\|Lu(t)\|_{L^ql^r}\le C\|\psi\|_{l^2}+C\|u^{2\sigma+1}\|_{L^1l^2}, \\
         &\le C\|\psi\|_{l^2}+C\|u\|^{2\sigma+1}_{L^{2\sigma+1}l^{2(2\sigma+1)}}.    \end{align*}
Since $\sigma\ge\frac{11}{6}$, we have
\[
\frac{1}{2\sigma+1}+\frac{3}{4(2\sigma+1)}\le\frac{3}{8}.
\]
Then we obtain
\begin{align*}
    \sup_{(q,r):\frac{1}{q}+\frac{3}{4r}\le\frac{3}{8}}\|Lu(t)\|_{L^ql^r}&\le C\|\psi\|_{l^2}+C\|u\|_{\mathcal{M}}^{2\sigma+1}\\
    &\le C\|\psi\|_{l^2}+C(2C\epsilon)^{2\sigma+1}.
\end{align*}
    Let $\epsilon$ be small enough, then we have 
    \[
    \sup_{(q,r):\frac{1}{q}+\frac{3}{4r}\le\frac{3}{8}}\|Lu(t)\|_{L^ql^r}\le2C\|\psi\|_{l^2}.
    \]

Therefore, the image $L(\mathcal{M})$ is contained in $\mathcal{M}$ given that $\|\psi\|$ is small enough.
Let $u,\hat{u}\in \mathcal{M}$, then with the same initial data $u(\cdot,0)=\hat{u}(\cdot,0)=\psi,$
\begin{align*}
    \sup_{(q,r):\frac{1}{q}+\frac{3}{4r}\le\frac{3}{8}}\|Lu-L\hat{u}\|&\le C_0\|u-\hat{u}\|_{L^{2\sigma+1}l^{2(2\sigma+1)}}(\|u\|_{L^{2\sigma+1}l^{2(2\sigma+1)}}+\|\hat{u}\|_{L^{2\sigma+1}l^{2(2\sigma+1)}}),\\
&\le2CC_1\|u-\hat{u}\|_{L^{2\sigma+1}l^{2(2\sigma+1)}}\|\psi\|_{l^2}^{2\sigma}.
\end{align*}
Therefore $L$ is a contraction map given that $\psi$ is sufficiently small.
Thus we obtain a global solution of \eqref{DNLS}.
\end{proof}

\section{Appendix}\label{app}
\begin{prop}\label{calculation}
	Let $\Sigma_1^2$ and $\Sigma_2'$ be the curves defined in Section \ref{sec:4}, then $\Sigma_1^2\cap\Sigma_2'=\{(k_1\pi+\frac{\pi}{2},k_2\pi+\frac{\pi}{2}):k_1,k_2\in\mathbb{Z}\}$.
\end{prop}
\begin{proof}
	Let $A=\frac{x_1-x_2}{2}$ and $B=\frac{x_1+x_2}{2}$. Then by the change of variables, we can rewrite \eqref{eq2} and \eqref{cancel_eq3} as
	\begin{align}\label{change1}
		8\cos^3A\cos^3B-6\cos^2A\cos^2B-2\cos^4B-\cos^2A+\cos^2B+1=0,
	\end{align}
and 
\begin{align}\label{change2}
	\cos^2 B(\cos^2A+\cos^2B-1)^2=\sin^2B(\cos^2B-\cos^2A+1)^2,
\end{align}
respectively. When $\cos^2B=\frac{1}{2},$ by equation \eqref{change2} we have $\cos^2 A=1,$ which means $(x_1,x_2)\in\Sigma_0\cap\Sigma_1^1.$ Then by equation \eqref{eq1}, the case will be attributed to Case (A) or Case (B). So without loss of generality, we assume that $\cos^2B\neq \frac{1}{2}$.
By direct calculation,
we have
\begin{align*}
	\sin^2A=\frac{\cos^2B(1\pm|\sin 2B|)}{2\cos^2B-1}
\end{align*}
Since $\sin^2A\in[0,1]$,
\begin{align}\label{cosa}
	\sin^2A=\frac{\cos^2B(1-|\sin 2B|)}{2\cos^2B-1}
\end{align}
Due to periodicity, we assume that $(A,B)\in(-\frac{\pi}{2},\frac{3\pi}{2})\times(-\frac{\pi}{2},\frac{\pi}{2}).$
For $(A,B)\in(-\frac{\pi}{2},\frac{\pi}{2})^2$, we have $(x_1,x_2)\in(-\pi,\pi)^2.$ Since $A\in(-\frac{\pi}{2},\frac{\pi}{2})$, $\cos{x_1}+\cos{x_2}>0.$ Therefore, $\cos^2x_1+\cos^2x_2-\cos{x_1}\cos{x_2}\cos(x_1-x_2)>0,$ so the equation \eqref{eq2} does not have a solution in this case. Therefore, we only need to consider the case when $(A,B)\in(\frac{\pi}{2},\frac{3\pi}{2})\times(-\frac{\pi}{2},\frac{\pi}{2}).$ 
Now we prove that if the equation \eqref{cosa} holds, then 
\[
8\cos^3A\cos^3B-6\cos^2A\cos^2B-2\cos^4B-\cos^2A+\cos^2B+1<0.
\]
We divide our proof into two parts. 
\begin{enumerate}[(I)]
	\item Firstly, we consider the case when $\cos^2B>\frac{1}{2}.$ Since $\cos A<0,$ we only need to prove 
	\begin{align}\label{caseI}
		-2\cos^4B-\cos^2A+\cos^2B+1=\cos^2B(1-2\cos^2B)+\sin^2A<0
	\end{align}
By \eqref{cosa}, we have
\begin{align*}
	\cos^2B(1-2\cos^2B)+\sin^2A=\frac{\cos^2B(\sin^22B-|\sin 2B|)}{2\cos^2B-1}<0.
\end{align*}
Therefore, we finish the proof of case (I).
\item When $\cos^2B<\frac{1}{2},$ by direct calculation, 
\begin{align*}
&-2\cos^2A\cos^2B-2\cos^4B-\cos^2A+\cos^2B+1
\\&=\frac{\cos^2B(\sin^22B-|\sin 2B|+2\sin^2B-2\cos^2B|\sin2B|)}{2\cos^2B-1}.
\end{align*}
Since $\cos^2B<\frac{1}{2}$, we have 
\[
2\sin^2B-|\sin2B|=2|\sin B|(|\sin B|-|\cos B|)>0.
\]
Since $\sin^2B>\frac{1}{2},$ we have
\[
\sin^22B-2\cos^2B|\sin 2B|=2\cos^22B(2\sin^2B-|\sin 2B|)>0.
\]
Therefore, we have $$-2\cos^2A\cos^2B-2\cos^4B-\cos^2A+\cos^2B+1<0.$$ Thus we finish the proof.
\end{enumerate}
\end{proof}

\textbf{Acknowledgements.} H. Ge is supported by NSFC, no.12341102, no.12122119. B. Hua is supported by NSFC, no.12371056, and by Shanghai Science and Technology Program [Project No. 22JC1400100]. 	 The authors would like to thank Jiawei Cheng for helpful advice.
\bibliographystyle{plain}
\bibliography{reference}

\noindent Huabin Ge, hbge@ruc.edu.cn\\[2pt]
\emph{School of Mathematics, Renmin University of China, Beijing, 100872, P.R. China} 
\\

\noindent Bobo Hua, bobohua@fudan.edu.cn\\[2pt]
\emph{School of Mathematical Sciences, LMNS, Fudan University, Shanghai, 200433, P.R. China}
\\

\noindent Longsong Jia, jialongsong@stu.pku.edu.cn\\[2pt]
\emph{School of Mathematical Sciences, Peking University, Beijing, 100871, P.R. China}
\\

\noindent Puchun Zhou, pczhou22@m.fudan.edu.cn\\[2pt]
\emph{School of Mathematical Sciences, Fudan University, Shanghai, 200433, P.R. China}

\end{document}